\documentclass[11pt, reqno]{amsart}
\usepackage{amssymb,amsthm, amsmath, amsfonts}
\usepackage{graphics}
\usepackage{hyperref}
\usepackage[all]{xy}
\usepackage{enumerate}
\usepackage{mathrsfs}

\theoremstyle{plain}
\newtheorem{theorem}{Theorem}
\newtheorem{lemma}[theorem]{Lemma}
\newtheorem{proposition}[theorem]{Proposition}
\newtheorem{corollary}[theorem]{Corollary}
\theoremstyle{definition}

\newtheorem{remark}[theorem]{Remark}

\DeclareMathOperator{\Hom}{Hom}
\DeclareMathOperator{\Tor}{Tor}
\DeclareMathOperator{\reg}{reg}

\newcommand{\N}{{\mathbb{N}}}
\newcommand{\Z}{{\mathbb{Z}}}
\newcommand{\FI}{{\mathrm{FI}}}
\newcommand{\wS}{{\widetilde{\mathit{S}}}}

\title{A remark on FI-module homology}
\author{Wee Liang Gan}
\address{Department of Mathematics, University of California, Riverside, CA 92521, USA}
\email{wlgan@math.ucr.edu}

\author{Liping Li}
\address{College of Mathematics and Computer Science; Performance Computing and Stochastic Information Processing (Ministry of Education), Hunan Normal University; Changsha, Hunan 410081, China.}
\email{lipingli@hunnu.edu.cn}

\thanks{The second author is supported by the National Natural Science Foundation of China 11541002, the Construct Program of the Key Discipline in Hunan Province, and the Start-Up Funds of Hunan Normal University 830122-0037. We thank the referee for many helpful suggestions that improved the exposition of the paper.}

\begin{document}

\begin{abstract}
We show that the FI-homology of an FI-module can be computed via a Koszul complex. As an application, we prove that the Castelnuovo-Mumford regularity of a finitely generated torsion FI-module is equal to its degree.
\end{abstract}

\maketitle

\section{Introduction} \label{introduction}

Let $R$ be a commutative ring. Let $\FI$ be the category of finite sets and injective maps. An $\FI$-module is a (covariant) functor from $\FI$ to the category of $R$-modules.  In \cite[(11)]{CEFN}, Church, Ellenberg, Farb, and Nagpal defined, for any $\FI$-module $V$, a complex $\wS_{-*}V$ of $\FI$-modules. The same complex was also considered independently by Putman \cite[\S4]{Putman}, who called it the central stability chain complex. The purpose of our present paper is to give a proof that the homology of the complex $\wS_{-*}V$ is the $\FI$-homology $H^{\FI}_*(V)$ of $V$ whose definition we now recall.

Let $\N$ be the set of non-negative integers. For each $n\in \N$, let $[n]$ be the set $\{1,\ldots,n\}$ with $n$ elements; in particular, $[0]=\emptyset$. It is convenient to introduce the non-unital $R$-algebra
\begin{equation*}
 A := \bigoplus_{0\leqslant m\leqslant n} A_{m,n}
\end{equation*}
where $A_{m,n}$ is the free $R$-module on the set $\Hom_{\FI}([m],[n])$. For any $\alpha\in \Hom_{\FI}([m],[n])$ and $\beta\in \Hom_{\FI}([r],[s])$, their product $\alpha\beta$ in $A$ is defined by
\begin{equation*}
\alpha\beta := \left\{ \begin{array}{ll}
\alpha\circ \beta & \mbox{ if } s=m, \\
0 & \mbox{ else. }
\end{array} \right.
\end{equation*}

Define a two-sided ideal $J$ of $A$ by
\begin{equation*}
J := \bigoplus_{0\leqslant m<n} A_{m,n}.
\end{equation*}

For each $n\in\N$, let $e_n\in A_{n,n}$ be the identity endomorphism of $[n]$. A left $A$-module $M$ is said to be \emph{graded} if $M=\bigoplus_{n\in\N} e_n M$. If $V$ is an $\FI$-module, then $\bigoplus_{n\in\N} V([n])$ has the natural structure of a graded left $A$-module. The functor $V \mapsto \bigoplus_{n\in\N} V([n])$ is an equivalence from the category of $\FI$-modules to the category of graded left $A$-modules. Thus, we shall not distinguish between the notion of $\FI$-modules and the notion of graded left $A$-modules. By abuse of notation, we shall also denote $\bigoplus_{n\in\N} V([n])$ by $V$.

Let $a\in \N$. Following \cite{CE} and \cite[Definition 2.3.7]{CEF}, the \emph{$\FI$-homology} functor $H^{\FI}_a$ is, by definition, the $a$-th left-derived functor of the right exact functor
\begin{equation} \label{H0}
H^{\FI}_0: V \mapsto V/JV
\end{equation}
from the category of $\FI$-modules to itself. In other words, $H^{\FI}_a(V) := \Tor^A_a(A/J,V)$.

The main result of this paper is the following.
\begin{theorem} \label{main result}
Let $V$ be an $\FI$-module. For each $a\in\N$, one has:
\begin{equation*}
H^{\FI}_a(V) \cong H_a(\wS_{-*}V).
\end{equation*}
\end{theorem}

Let us briefly describe our strategy for proving Theorem \ref{main result}. For each $d\in \N$ and finite set $X$, let $M(d)(X) := R \Hom_{\FI}([d],X)$, the free $R$-module on the set of injections from $[d]$ to $X$. This defines an $\FI$-module $M(d)$. We shall prove that for each $d\in \N$, one has $H_a(\wS_{-*} M(d))=0$ whenever $a\geqslant 1$. This key result implies that for $a\geqslant 1$, the functor $V \mapsto H_a(\wS_{-*}V)$ is coeffaceable because every $\FI$-module is a quotient of $\bigoplus_{i\in I} M(d_i)$ for some $d_i\in \N$. It is easy to see that the collection of functors $V \mapsto H_a(\wS_{-*}V)$ for $a\geqslant 0$ form a homological $\delta$-functor with $H_0(\wS_{-*}V)\cong V/JV$, and hence Theorem \ref{main result} follows by a standard result in homological algebra.

We give an application of Theorem \ref{main result}.

For any $\FI$-module $V$, its \emph{degree} is defined to be
\begin{equation*}
\deg(V) := \sup\{ n\in \N \mid V([n]) \neq 0 \},
\end{equation*}
where by convention the supremum of an empty set is $-\infty$. The \emph{Castelnuovo-Mumford regularity} $\reg(V)$ of $V$ is defined to be the infimum of the set of all $c\in\Z$ such that

\begin{equation*}
\deg(H^{\FI}_a(V)) \leqslant c+a \quad \mbox{ for each } a\geqslant 1,
\end{equation*}
where by convention the infimum of an empty set is $\infty$.

\begin{theorem} \label{reg of torsion module}
Let $V$ be a finitely generated $\FI$-module with $\deg(V)<\infty$. Then
\begin{equation*}
\reg(V) = \deg(V).
\end{equation*}
\end{theorem}

After recalling the construction of the complex $\wS_{-*}V$ in the next section, we prove the key result we need in Section \ref{key fact}. The proof of Theorem \ref{main result} will be given in Section \ref{the proofs}, and the proof of Theorem \ref{reg of torsion module} will be given in Section \ref{an application}.

\begin{remark}
Theorem \ref{main result} was first stated without proof in \cite[Remark 2.21]{CEFN}; after our present paper was written, we learnt that a different proof was given by Church and Ellenberg in \cite[Proposition 4.9]{CE}. Theorem \ref{reg of torsion module} was first conjectured in \cite[Corollary 5.22]{LR}.
\end{remark}

\section{Reminder on Koszul complex of an FI-module} \label{reminder}

The complex $\wS_{-*}V$ defined in \cite[(11)]{CEFN} is a Koszul complex \cite[\S12.4]{KS}. To state its definition, we need a few notations.

For any finite set $I$, let $RI$ be the free $R$-module with basis $I$, and $\det(I)$ the free $R$-module $\bigwedge^{|I|} RI$ of rank one. In particular, if $I=\emptyset$, then $\det(I)=R$. If $I=\{i_1,\ldots,i_a\}$, then $i_1\wedge \cdots \wedge i_a$ is a basis for $\det(I)$.

Let $V$ be an $\FI$-module and $X$ a finite set. Suppose $T\subset X$ and $i\in X\setminus T$. For any $v\in V(T)$, we shall write $i(v)$ for the element $\iota(v)\in V(T\cup\{i\})$ where $\iota$ is the inclusion map from $T$ to $T\cup \{i\}$. For any $a\in \N$, let
\begin{equation} \label{terms of Koszul complex}
(\wS_{-a}V)(X) := \bigoplus_{\substack{I\subset X \\ |I|=a}} V(X\setminus I)\otimes_R \det(I).
\end{equation}
The differential $d:(\wS_{-a}V)(X) \to (\wS_{-a+1}V)(X)$ is defined on each direct summand by the formula
\begin{equation*}
d(v\otimes i_1\wedge \cdots \wedge i_a) := \sum_{p=1}^a (-1)^p i_p(v) \otimes i_1\wedge\cdots \widehat{i_p} \cdots\wedge i_a,
\end{equation*}
where $v\in V(X\setminus I)$, $I=\{i_1,\ldots, i_a\}$, and $\widehat{i_p}$ means that $i_p$ is omitted in the wedge product. If $X'$ is a finite set and $\alpha: X\to X'$ is an injective map, then $\alpha: (\wS_{-a}V)(X) \to (\wS_{-a}V)(X')$ is defined on each direct summand by the formula
\begin{equation*}
\alpha(v \otimes i_1\wedge \cdots \wedge i_a) := \alpha(v) \otimes \alpha(i_1)\wedge \cdots \wedge\alpha(i_a),
\end{equation*}
where $v\in V(X\setminus I)$, $\alpha(v)\in V(X'\setminus \alpha(I))$, and $I=\{i_1,\ldots, i_a\}$. This defines, for each $a\in \N$, an $\FI$-module $\wS_{-a}V$. Thus, we obtain a complex $\wS_{-*}V$ of $\FI$-modules:
\begin{equation} \label{complex}
 \cdots \longrightarrow \wS_{-2}V \longrightarrow \wS_{-1}V \longrightarrow \wS_{0}V \longrightarrow 0.
 \end{equation}
For each $a\in  \N$, the homology $H_a(\wS_{-*}V)$ of the complex (\ref{complex}) is an $\FI$-module.

Observe that for each $n\in \N$, the image of $d: (\wS_{-1} V)([n]) \to (\wS_0 V)([n])$ is the $R$-submodule of $V([n])$ spanned by all elements of the form $\alpha(v)$ for some injective map $\alpha: [n-1] \to [n]$ and $v\in V([n-1])$, so
\begin{equation} \label{0-th homology of Koszul complex}
H_0(\wS_{-*} V) \cong V/JV.
\end{equation}

\begin{remark}
In \cite{CEFN}, the complex $\wS_{-*}V$ is constructed via another complex $B_*V$ defined in \cite[(10)]{CEFN}. We have given a direct construction of $\wS_{-*}V$.
\end{remark}

\section{A key fact} \label{key fact}

Recall that for each $d\in \N$, the $\FI$-module $M(d)$ is defined by $M(d)(X) := R \Hom_{\FI}([d],X)$ for each finite set $X$.

\begin{lemma} \label{degree 0}
Let $d\in \N$. If $n\neq d$, then $H_0( (\wS_{-*} M(d))([n]) )=0$.
\end{lemma}

\begin{proof}
The complex  $(\wS_{-*} M(d))([n])$ is:
\begin{equation*}
 \cdots \longrightarrow \bigoplus_{\substack{T\subset [n] \\ |T|=n-1}} R\Hom_{\FI}([d],T) \otimes_R \det( [n]\setminus T ) \longrightarrow R\Hom_{\FI}([d], [n]) \longrightarrow 0.
\end{equation*}
If $n<d$, then $\Hom_{\FI}(d, [n])=\emptyset$. If $n>d$, the image of any injective map from $[d]$ to $[n]$ lies in some subset $T\subset [n]$ with $|T|=n-1$. The lemma follows.
\end{proof}

The following result plays a crucial role in this paper.

\begin{proposition} \label{exactness}
Let $d\in \N$. If $a \geqslant 1$, then $H_a( \wS_{-*} M(d) ) = 0$.
\end{proposition}

\begin{proof}
We have to show that for each $n\in \N$, the homology $H_a((\wS_{-*} M(d))([n]))$ is 0 for all $a\geqslant 1$. We do this by induction on $n$. If $a>n$, then for any $\FI$-module $V$ one has $(\wS_{-a}V)([n])=0$. Therefore, if $a>n$, then $(\wS_{-a} M(d))([n])$ is 0 for all $d\in \N$. Thus,  the $n=0$ case is trivial. Now suppose $n>0$, and $H_a((\wS_{-*} M(d))([n-1]))=0$ for all $a\geqslant 1$ and for all $d\in \N$.

For each $a\in \N$, let $C_{-a}$ be the $R$-submodule of $(\wS_{-a} M(d))([n])$ spanned by the elements $\alpha \otimes i_1 \wedge \cdots \wedge i_a$ where $\alpha$ is an injective map from $[d]$ to $[n]\setminus \{i_1,\ldots, i_a\}$ and $n \notin \{i_1,\ldots, i_a\}$; thus, $C_{-a}$ is spanned by the direct summands of (\ref{terms of Koszul complex}) for $V=M(d)$ and $X=[n]$ with $n\notin I$. (Note that the definition of $C_{-a}$ depends on $d$ and $n$ but we suppress them from our notation. Moreover, $C_{-a}$ is not an $S_n$-submodule of $(\wS_{-a} M(d))([n])$.) It is clear that $C_{-*}$ is a subcomplex of $(\wS_{-*} M(d))([n])$.

For each basis element $\alpha  \otimes i_1 \wedge \cdots \wedge i_a  \in C_{-a}$, we have either $n \notin \alpha([d])$, or $n=\alpha(r)$ for exactly one $r\in [d]$; if $n\notin \alpha([d])$, then we can consider $\alpha$ as an injective map from $[d]$ to $[n-1]\setminus\{i_1,\ldots, i_a\}$ and so $\alpha  \otimes i_1 \wedge \cdots \wedge i_a$ can be considered as an element of $(\wS_{-a} M(d))([n-1])$, whereas if $r\in [d]$ and $n=\alpha(r)$, then $\alpha\vert_{[d]\setminus \{r\}}$ can be considered as an injective map from $[d-1]$ to $[n-1]\setminus \{i_1,\ldots, i_a\}$ (by fixing an identification of $[d]\setminus \{r\}$ with $[d-1]$) and so $\alpha\vert_{[d]\setminus \{r\}} \otimes i_1 \wedge \cdots \wedge i_a$ can be considered as an element of $(\wS_{-a} M(d-1))([n-1])$. Hence, we can identify the complex $C_{-*}$ with:

\begin{equation*}
(\wS_{-*} M(d))([n-1]) \bigoplus \left( \bigoplus_{r=1}^d (\wS_{-*} M(d-1))([n-1]) \right).
\end{equation*}

The quotient complex of $(\wS_{-*} M(d))([n])$ by $C_{-*}$ can be identified with $\Sigma (\wS_{-*} M(d))([n-1])$, where $\Sigma$ is the shift functor on complexes so that $(\Sigma K)_{-a} = K_{-a+1}$ for any complex $K$; specifically, each basis element $\alpha\otimes i_1\wedge \cdots \wedge i_{a-1}$ of $(\wS_{-a+1} M(d))([n-1])$ (where $\alpha$ is an injective map from $[d]$ to $[n-1]\setminus \{i_1,\ldots, i_{a-1}\}$) can be identified with the element $\alpha\otimes i_1\wedge \cdots \wedge i_{a-1} \wedge n \in (\wS_{-a} M(d))([n])$.

Therefore, we have a short exact sequence of complexes:
\begin{equation*}
0 \longrightarrow C_{-*} \longrightarrow (\wS_{-*} M(d))([n]) \longrightarrow \Sigma (\wS_{-*} M(d))([n-1]) \longrightarrow 0.
\end{equation*}

This gives a long exact sequence in homology:
\begin{multline*}
\cdots \longrightarrow H_a(C_{-*}) \longrightarrow H_a( (\wS_{-*} M(d))([n]) ) \longrightarrow H_{a-1} ((\wS_{-*} M(d))([n-1])) \\
\longrightarrow H_{a-1} (C_{-*}) \longrightarrow \cdots .
\end{multline*}
By the induction hypothesis, if $a\geqslant 2$, then we have $H_a(C_{-*})=0$ and $H_{a-1}((\wS_{-*} M(d))([n-1]))=0$, so $H_a( (\wS_{-*} M(d))([n]) ) =0$. By the induction hypothesis, we also have $H_1(C_{-*})=0$. Moreover, by Lemma \ref{degree 0}, we have $H_0 ((\wS_{-*} M(d))([n-1]))=0$ if $d\neq n-1$. Hence we have $H_1( (\wS_{-*} M(d))([n]) ) =0$ if $d\neq n-1$. If $d=n-1$, the complex $(\wS_{-*} M(d))([n]) )$ is:
\begin{equation*}
0 \longrightarrow \bigoplus_{\substack{T\subset [n] \\ |T|=n-1}} R\Hom_{\FI}([n-1],T) \otimes_R \det( [n]\setminus T ) \longrightarrow R\Hom_{\FI}([n-1], [n]) \longrightarrow 0.
\end{equation*}
Since the above sequence is exact, we have $H_1( (\wS_{-*} M(n-1))([n]) ) =0$.
\end{proof}

\section{Proof of Theorem \ref{main result}} \label{the proofs}

We will now prove Theorem \ref{main result}.

\begin{proof}[Proof of Theorem \ref{main result}]
For any $\FI$-module $V$ and $a\in \N$, we shall write
\begin{equation*}
\widetilde{H}_a(V) := H_a(\wS_{-*}V);
\end{equation*}
our aim is to prove that $H^{\FI}_a(V) \cong \widetilde{H}_a(V)$.

Any short exact sequence $0 \to V' \to V \to V'' \to 0$ of $\FI$-modules gives a short exact sequence
$0 \to \wS_{-*} V'\to  \wS_{-*} V \to  \wS_{-*} V'' \to 0$ of complexes. It follows that the system $\widetilde{H} = (\widetilde{H}_*)$ is a homological $\delta$-functor. For any $\FI$-module $V$, there is a surjective homomorphism $\bigoplus_{i\in I} M(d_i) \to V$ for some $d_i\in \N$. Hence, by Proposition \ref{exactness}, the functor $\widetilde{H}_a$ is coeffaceable for each $a\geqslant 1$. By \cite[Proposition 2.2.1]{Grothendieck}, it follows that $\widetilde{H}$ is a universal $\delta$-functor. Since, by (\ref{H0}) and (\ref{0-th homology of Koszul complex}), we have
\begin{equation*}
\widetilde{H}_0 (V) \cong \frac{V}{JV} \cong H^{\FI}_0(V),
\end{equation*}
it follows that $\widetilde{H}_a (V) \cong H^{\FI}_a(V)$ for all $a$.
\end{proof}

The following corollary was first proved in \cite[Proposition 2.25]{CEFN}; let us show that it is an immediate consequence of Theorem \ref{main result} and \cite[Theorem A]{CEFN}.

\begin{corollary} \label{vanishing}
Suppose $R$ is noetherian and $V$ is a finitely generated $\FI$-module. For each $a\in \N$, one has $\deg( H_a(\wS_{-*}V) ) < \infty$.
\end{corollary}

\begin{proof}[Proof of Corollary \ref{vanishing}]
By \cite[Theorem A]{CEFN}, there is a projective resolution $\cdots \to P_2 \to P_1 \to P_0$ of $V$, where each $P_a$ is a finitely generated projective $\FI$-module. Applying the functor (\ref{H0}) to the complex $P_*$, we obtain a complex $P_*/JP_*$ where each $P_a/JP_a$ is a finitely generated $\FI$-module. But for any finitely generated $\FI$-module $W$ with $JW=0$, one has $W([n]) = 0$ for all $n$ sufficiently large; indeed, if $W$ is generated by a finite set of elements $w_i \in W([n_i])$ for $i=1,\ldots, m$, then $W([n])=0$ for $n>\max\{n_1,\ldots, n_m\}$ since $\alpha(w_i)=0$ for every $\alpha\in \Hom_{\FI}([n_i], [n])$ with $n>n_i$. Therefore, we have $\deg(P_a/JP_a) < \infty$. By Theorem \ref{main result}, we have $H_a (\wS_{-*}V) \cong H^{\FI}_a(V) \cong H_a ( P_*/JP_* )$. Hence, $\deg(H_a(\wS_{-*}V)) < \infty$.
\end{proof}

\section{An application} \label{an application}

The goal of this section is to prove Theorem \ref{reg of torsion module}.

\begin{proof}[Proof of Theorem \ref{reg of torsion module}]
The $V=0$ case is trivial, so assume that $V\neq 0$. Set $p=\deg(V)$.

Let $a\in \N$. If $n>p+a$, then $V([n-a])=0$, so $(\wS_{-a}V)([n]) = 0$, and hence by Theorem \ref{main result} one has $H^{\FI}_a(V)([n])=0$. We claim that $H^{\FI}_a(V)([p+a]) \neq 0$ for all $a$ sufficiently large; this would imply the Theorem.

One has:
\begin{align*}
(\wS_{-(a+1)} V)([p+a]) &\cong V([p-1])^{\oplus {p+a \choose a+1}},\\
(\wS_{-a} V)([p+a]) &\cong V([p])^{\oplus {p+a \choose a}},\\
(\wS_{-(a-1)} V)([p+a]) &\cong 0,
\end{align*}
where $V([p-1])$ is 0 if $p=0$. By Theorem \ref{main result}, to prove our claim, it suffices to show the following statement: \emph{for all $a$ sufficiently large, there is no surjective $R$-linear map \[ V([p-1])^{\oplus {p+a \choose a+1}} \longrightarrow V([p])^{\oplus {p+a \choose a}}.\]}

If $R$ is a field, the above statement is clear for $\dim\left( V([p-1])^{\oplus {p+a \choose a+1}} \right)$ is a polynomial in $a$ of degree at most $p-1$, while $\dim \left( V([p])^{\oplus {p+a \choose a}} \right) $ is a polynomial in $a$ of degree $p$. Suppose now that $R$ is any commutative ring and set $M=V([p])$; we can reduce the statement to the case where $R$ is a field using the following observations:
\begin{itemize}
\item Since $M$ is a nonzero $R$-module, there exists a maximal ideal $\mathfrak{q}$ of $R$ such that the localization $M_{\mathfrak{q}}$ is nonzero; thus we reduce to the case where $R$ is a ring.

\item If $R$ is a local ring with unique maximal ideal $\mathfrak{q}$, then since $M$ is a nonzero finitely generated $R$-module, it follows by Nakayama's lemma that $M/\mathfrak{q}M$ is nonzero; thus we reduce to the case where $R$ is a field.
\end{itemize}

This concludes the proof of the theorem.
\end{proof}

We believe that the conclusion of this theorem holds for torsion modules which might not be finitely generated. However, the technique used here (comparing dimensions) does not work in this general case. It would be interesting to find a proof valid for all torsion modules.

\end{document}